\documentclass[a4paper,12pt]{amsart} 

\usepackage{amssymb,amsmath}
\usepackage{amscd}
\usepackage{amsthm}
\usepackage{mathrsfs}

\newtheorem{definition}{Definition}[section]
\newtheorem{theorem}{Theorem}[section]
\newtheorem{lemma}{Lemma}[section]

\newtheorem{cor}{Corollary}[section]
\newtheorem{fact}{Fact}[section]
\newtheorem{remark}{Remark}[section]

\def\bbp{\mathbb{P}}

\def\bbr{\mathbb{R}}
\def\bbk{\mathbb{K}}
\def\bbl{\mathbb{L}}

\def\ca{\mathscr{A}}
\def\cb{\mathscr{B}}

\def\cf{\mathscr{F}}
\def\ci{\mathscr{I}}
\def\cj{\mathscr{J}}
\def\cp{\mathscr{P}}
\def\cx{\mathscr{X}}

\def\Borel{\mbox{\rm Borel}}
\def\Open{\mbox{\rm Open}}

\def\<{{\langle}}
\def\>{\rangle}

\def\lbv{{[\!|}}
\def\rbv{{|\!]}}

\def\force{{\;\Vdash\;}}

\def\iff{\longleftrightarrow}

\title{Generalized Luzin sets}
\author{Robert Ra{\l}owski and Szymon \.Zeberski}
\address{ Institute of 
         Mathematics and Computer Science, 
         Wroc{\l}aw University of Technology, Wybrze\.ze Wyspia\'n\-skie\-go 27, 
         50-370 Wroc{\l}aw, Poland.}

\email[Robert Ra{\l}owski]{robert.ralowski@pwr.wroc.pl}
\email[Szymon \.Zeberski]{szymon.zeberski@pwr.wroc.pl}


\date{}

\begin{document}
\maketitle
\begin{abstract}
 In this paper we invastigate the notion of generalized $(\ci,\cj)$ - Luzin set. 
 This notion generalize the standard notion of Luzin set and Sierpi{\'n}ski set. 
 We find set theoretical conditions which imply the existence of generalized $(\ci, \cj)$ - Luzin set. 
 We show how to construct large family of pairwise non-equivalent $(\ci,\cj)$ - Luzin sets. 
 We find a class of forcings which preserves the property of being $(\ci,\cj)$ - Luzin set. 
\end{abstract}

\section{Notation and Terminology}

We will use standard set-theoretic notation following \cite{jech}. In particular for any set $X$ and any cardinal $\kappa$,
  $[X]^{<\kappa}$ denotes the set of all subsets of $X$ with size less than $\kappa.$ Similarly, $[X]^\kappa$ denotes the family of subsets of $X$ of size $\kappa.$ By $\cp(X)$ we denote the  power set of $X$.

If $A\subseteq X\times Y$ then for $x\in X$ and $y\in Y$ we put
$$
A_x=\{y\in Y:\; (x,y)\in A\},
$$
$$
A^y=\{x\in X:\; (x,y)\in A\}.
$$
By $A\vartriangle B$ we denote the symmetric difference of sets $A$ and $B,$ i.e.
$$
A\vartriangle B= (A\setminus B)\cup (B\setminus A).
$$ 
In this paper $\cx$ denotes uncountable Polish space. 
By $\Open(\cx)$ we denote the topology of $\cx.$
By $\Borel(\cx)$ we denote the $\sigma$-field of all Borel sets. Let us recall that each Borel set can be coded by a function from $\omega^\omega.$ Precise definition of such coding can be found in \cite{kechris}. If $x\in\omega^\omega$ is a Borel code then by $\#x$ we denote the Borel set coded by $x.$ 

$\ci$, $\cj$ are $\sigma $-ideals on $\cx,$ i.e. $\ci, \cj\subseteq\cp(\cx)$ are closed under countable unions and subsets. Additionally we assume that $[\cx]^\omega\subseteq\ci, \cj.$ Moreover $\ci, \cj$ have Borel base i.e each set from the ideal can be covered by a Borel set from the ideal. Standard examples of such ideals are the ideal $\bbl$ of Lebesgue measure zero sets and the ideal $\bbk$ of meager sets of Polish space.

\begin{definition} Let $M\subseteq N$ be standard transitive models of ZF. \\
Coding Borel sets from the ideal $I$ is absolute iff
$$
(\forall x\in M\cap \omega^\omega)( M\models \#x\in I\iff N\models \#x\in I).
$$
\end{definition}

We say that $\ci$ satisfies $\kappa$ chain condition ($\kappa$-c.c.) if every family $\ca$ of Borel subsets of $\cx$ satisfying the following conditions:
\begin{enumerate}
 \item $(\forall A\in\ca)(A\notin\ci)$
 \item $(\forall A,B\in\ca)(A\neq B\rightarrow A\cap B\in\ci)$
\end{enumerate}
 has size smaller than $\kappa.$
If $\ci$ is $\omega_1$-c.c. then we say that $\ci$ is c.c.c.

Let us recall that a function $f:\cx\rightarrow\cx $ is $\ci$-measurable if the preimage of 
every open subset of $\cx$ is $\ci$-measurable i.e belongs to the $\sigma $-field generated by 
Borel sets and the ideal $\ci.$ In other words $f$ is $\ci$-measurable iff
$$
(\forall U\in \Open(\cx))(\exists B\in \Borel(\cx))(\exists I\in\ci)(f^{-1}[U]=B\vartriangle I).
$$ 

Let us recall the following cardinal coefficients:
\begin{definition}[Cardinal coefficients] 
$$\begin{array}{l@{\ =\,{} }l}
non(\ci) &  \min\{ |A|:\; A\subseteq \cx\land A\notin \ci\}\\
add(\ci) &  \min\{ |\ca|:\;\ca\subseteq \ci\land \bigcup \ca\notin \ci \}\\
cov(\ci) &  \min\{ |\ca|:\;\ca\subseteq \ci\land \bigcup \ca=\cx \}\\
cov_h(\ci) &  \min\{ |\ca|:\;\ca\subseteq \ci\land (\exists B\in \Borel(\cx)\setminus\ci)( B\subseteq  
       \bigcup \ca) \}\\
cof(\ci) & \min\{ |\ca|:\; \ca\subseteq \ci\land \ca\text{ is a base of }\ci\}
\end{array}$$
where $\ca\text{ is a base of }\ci$ iff $\ca\subseteq\ci\land (\forall I\in\ci)(\exists A\in\ca)( I\subseteq A).$
\end{definition}
Let us remark that above coefficients can be defined for larger class of families (not only ideals).

\begin{definition}
We say that $L\subseteq \cx$ is a $(\ci,\cj)$ - Luzin set if
\begin{itemize}
 \item $L\notin \ci,$ 
 \item $(\forall B\in \ci)( B\cap L\in \cj).$
\end{itemize}
Assume that $\kappa $ is a cardinal number. We say that $L\subseteq \cx$ is a $(\kappa,\ci,\cj)$ -  Luzin set iff $L$ is a $(\ci,\cj)$ - Luzin set and $|L|=\kappa$.
\end{definition}

The above definition generalizes the standard notion of Luzin and Sierpi{\'n}ski sets. Namely, $L$ is Luzin set iff $L$ is generalized $(\bbl,[\bbr]^{\le\omega})$ - Luzin set and $S$ is Sierpi{\'n}ski set iff $S$ is generalized $(\bbk,[\bbr]^{\le\omega})$ - Luzin set.
The above notion generalizes also notions from \cite{C}.

\begin{definition}
We say that ideals $\ci$ and $\cj$ are orthogonal if
$$
\exists A\in\cp(\cx)\; A\in I\land A^c\in \cj.
$$
In such case we write $\ci\perp \cj$.
\end{definition}

\begin{definition} Let $\cf\subseteq \cx^\cx$ be a family of functions. We say that $A,B\subseteq \cx$ are equivalent with respect to $\cf$ if
$$
(\exists f\in \cf)\; (B=f[A]\vee A=f[B])
$$
\end{definition}

\begin{definition} We say that $A,B\subseteq \cx$ are Borel equivalent if
$A, B$ are equivalent with respect to the family of all Borel functions.
\end{definition}

\begin{definition} We say that $\ci$ has Fubini property iff for every Borel set $A\subseteq\cx\times\cx$
$$\{x\in\cx : A_x\notin\ci\}\in\ci\Longrightarrow\{y\in\cx : A^y\notin\ci\}\in\ci$$
\end{definition}
Natural examples of ideals fulfilling Fubini property are the ideal of null sets $\bbl$ (by Fubini theorem) and the ideal of meager sets $\bbk$ (by Kuratowski-Ulam theorem). 

By definition we can obtain the following properties:
\begin{fact} Assume that $\ci\perp \cj.$  
\begin{enumerate}
 \item There exist a $(\ci,\cj)$ - Luzin set. 
 \item If $L$ is a $(\ci,\cj)$ - Luzin set then $L$ is not $(\cj,\ci)$ - Luzin set.
\end{enumerate}
\end{fact}
\begin{proof}
 (Part 1) By the definition of $\ci\perp\cj$ we can find two sets $I\in\ci$ and $J\in\cj$ such that $I\cup J=\cx.$ We will show that $J$ is $(\ci,\cj)$ - Luzin set. 
 $J$ is not in $\ci.$ Let us fix any set $A\in\ci.$ We have that $A\cap J\subseteq J\in\cj.$

 (Part 2) By the definition of $\ci\perp\cj$ we can find two sets $I\in\ci$ and $J\in\cj$ such that $I\cup J=\cx.$ Assume that $L$ is $(\ci,\cj)$ - Luzin set and $(\cj,\ci)$ - Luzin set. We have that
$$
L\cap J\subseteq J\in \cj \text{ and } L\cap I\subset I\in \ci
$$
By the property of being $(\cj,\ci)$ - Luzin set
$$
L\cap J\in\ci.
$$
So
$L= (L\cap J)\cup (L\cap I)\in \ci.$ what is a contradiction with being $(\ci, \cj)$ - Luzin set.
\end{proof}

We will try to find a wide class of forcings which preserves the property of being $(\ci, \cj)$ - Luzin set. We are mainly interested in so called definable forcings (see \cite{zapletal}). Let us recall that $\bbp$ is definable forcing if $\bbp$ is of the form $\Borel(\cx)\setminus\ci,$ where $\cx$ and $\ci$ have absolute definition for standard transitive models of ZF of the same hight. 

\section{Existence of Luzin sets }

Let us start with a theorem which under suitable assumptions guarantees 
existence of uncountably many pairwise different $(\ci,\cj)$ - Luzin sets. 

\begin{theorem}
 Assume that $\kappa=cov(\ci)=cof(\ci)\le non(\cj).$
 Let $\cf $ be a family of functions from $\cx$ to $\cx.$ 
 Assume that $|\cf|\le\kappa.$
 Then we can find a sequence $(L_\alpha)_{\alpha<\kappa}$ such that
 \begin{enumerate}
  \item $L_\alpha$ is $(\kappa,\ci,\cj)$ - Luzin set,
  \item for $\alpha\neq\beta,\ L_\alpha$ is not equivalent to $L_\beta$ 
        with respect to the family $\cf.$   
 \end{enumerate}
\end{theorem}
\begin{proof} 
Let us enumerate the family $\cf$:
$$
\cf=\{ f_\alpha:\alpha<\kappa\}.
$$
Now, let us enumerate Borel base of ideal $\ci$:
$$
\cb_\ci=\{ B_\alpha:\;\alpha<\kappa\}.
$$

Now without loss of generality we can assume that 
$$
(\forall f\in\cf)(\forall \lambda<\kappa) ( \kappa\le |f[(\bigcup\nolimits_{\xi<\lambda}B_\xi)^c]| )
$$
Indeed, since $cov(\ci)=\kappa$ a set $(\bigcup\nolimits_{\xi<\lambda}B_\xi)^c$ is not in the ideal $\ci.$
If the function $f$ does not have the above property and $L$ is a $(\ci,\cj)$-Luzin set then 
$$
f[L]=f[L\cap \bigcup\nolimits_{\xi<\lambda}B_\xi]\cup f[L\cap (\bigcup\nolimits_{\xi<\lambda}B_\xi)^c]
$$
and both sets has cardinality less than $\kappa.$ So $f[L]$ is not $(\ci,\cj)$-Luzin set.

By induction we will construct the family $\{x^\eta_{\alpha,\zeta}: \eta,\zeta,\alpha<\kappa\}$ and  $\{d^\eta_{\alpha,\zeta}: \eta,\zeta,\alpha<\kappa\}$
such that 
$$
d^\eta_{\alpha,\zeta}=f_\zeta(x^\eta_{\alpha,\zeta})
$$
and for any different $\eta,\eta'<\kappa$
$$
\{x^\eta_{\alpha,\zeta}: \zeta,\alpha<\kappa\} \cap \{d^{\eta'}_{\alpha,\zeta}: \zeta,\alpha<\kappa\}=\emptyset
$$
and 
$$
x_{\alpha,\zeta}^\eta\in\cx\setminus \left(\{d^\eta_{\xi,\zeta}:\;\eta,\xi,\zeta<\alpha\}\cup \{x^\eta_{\xi,\zeta}:\;\eta,\xi,\zeta<\alpha\}\cup\bigcup_{\xi<\alpha }B_\xi\right)
$$
for every $\eta,\zeta<\alpha.$

Assume that we are in $\alpha$-th step of construction. Fix $\eta,\zeta<\alpha$. It means that we have constructed the following set
$$
Old=\{ x^\lambda_{\beta,\xi},d^\lambda_{\beta,\xi}:\; \beta,\xi,\lambda<\alpha\}\cup \{ x^\lambda_{\alpha,\xi},d^\lambda_{\alpha,\xi}:\; \lambda<\eta\lor (\lambda=\eta\land \xi<\zeta)\}.
$$
Since $|f_\zeta[(\bigcup_{\xi<\alpha} B_\xi)^c]|\ge\kappa$ and $|Old|<\kappa$ we get that 
$$
|f_\zeta[(\bigcup_{\xi<\alpha} B_\xi\cup Old)^c]|\ge\kappa.
$$
That's why we can find 
$$
d^\eta_{\alpha,\zeta}\in f_\zeta[(\bigcup_{\xi<\alpha} B_\xi\cup Old)^c]\setminus Old.
$$
Let $x^\eta_{\alpha,\zeta}$ be such that $d^\eta_{\alpha,\zeta}=f_\zeta(x^\eta_{\alpha,\zeta})$. In this way we can finish the $\alpha$-th step of construction.


Now, let us define $L_\alpha=\{x^\alpha_{\xi,\zeta} : \xi,\zeta<\kappa\}.$ 

Let us check that $L_\alpha$ is $(\ci,\cj)$ - Luzin set. Indeed,
if $A\in \ci$ then there exists $\beta<\kappa$ s.t. $A\subset B_\beta$.
Then we have
$$
A\cap L_\alpha\subset B_\beta\cap L_\alpha=
B_\beta\cap \{x^\alpha_{\xi,\zeta}:\;\xi,\zeta<\beta\}
\subseteq \{x^\alpha_{\xi,\zeta}:\;\xi,\zeta<\beta\}\in \cj
$$
because $|\{x^\alpha_{\xi,\zeta}:\;\xi,\zeta<\beta\}|\le|\beta|<\kappa\le non(\cj).$ 

What is more, for every function $f=f_\alpha\in\cf$ and every $\beta\neq\gamma$ 
we have that 
$$
\kappa\le |f[L_\gamma]\setminus L_\beta|
$$
because $\{ d^{\gamma}_{\xi,\alpha}:\alpha<\xi<\kappa\}\subseteq f[L_\gamma]\setminus L_\beta$.
So $L_\beta\neq f[L_\gamma].$
\end{proof}

In fact we haved proved a little stronger result.
\begin{remark}\label{stronger}
Assume that $\kappa=cov(\ci)=cof(\ci)\le non(\cj).$
 Let $\cf $ be a family of functions from $\cx$ to $\cx.$ 
 Assume that $|\cf|\le\kappa.$
 Then we can find a sequence $(L_\alpha)_{\alpha<\kappa}$ such that
 \begin{enumerate}
  \item $L_\alpha$ is $(\kappa,\ci,\cj)$ - Luzin set,
  \item for $\alpha\neq\beta$ and $f\in\cf$ we have that $\kappa\le |f[L_\alpha]\vartriangle L_\beta|.$
 \end{enumerate} 
\end{remark}

Let us notice that for every ideal $\ci$ we have the inequality 
$cov(\ci)\le cof(\ci).$ This gives
 the following corollary.
\begin{cor}\label{conti} 
If $2^\omega=cov(\ci)=non(\cj)$ 
then there exists continuum many different $(\ci,\cj)$ - Luzin sets 
which aren't Borel equivalent.

In particular, if CH holds then there exists continuum many different 
$(\omega_1,\ci,\cj)$ - Luzin sets which aren't Borel equivalent.
\end{cor}

We can extend above corollary to a wilder class of functions 
- namely, $\ci$-measurable functions. 

\begin{cor}\label{conti2}
 If $2^\omega=cov(\ci)=non(\cj)$ 
then there exists continuum many different $(\ci,\cj)$ - Luzin sets 
which aren't equivalent with respect to all $\ci$-measurable functions.

In particular, if CH holds then there exists continuum many different 
$(\omega_1,\ci,\cj)$ - Luzin sets which aren't  equivalent with 
respect to all $\ci$-measurable functions.
\end{cor}
\begin{proof}
 First, let us notice that if a function $f$ is $\ci$-measurable then there exists
 a set $I\in\ci\cap\Borel(\cx)$ such that $f\upharpoonright (\cx\setminus I)$ is Borel.
 Indeed, it is enough to consider a countable base $\{U_n\}_{n\in\omega}$ of 
 topology of $\cx.$ Then $f^{-1}[U_n]=B_n\vartriangle I_n,$ where $B_n$ is Borel 
 and $I_n$ is from the ideal $\ci.$ Now, put $I=\bigcup_{n\in\omega}I_n.$
 
 So we can consider a family of partial Borel functions which domain is Borel set 
 with complement in the ideal $\ci.$ This family is naturally of size continuum.
 So we can use Corollary \ref{conti} and Remark \ref{stronger} to finish the proof.
\end{proof}

Now, let us concentrate on ideal of null and meager sets.
\begin{cor} 
\begin{enumerate}
 \item Assume that $cov(\bbl)=2^\omega.$ 
       There exists continuum many different $(2^\omega,\bbl,\bbk)$ - Luzin sets which 
       aren't equivalent with respect to the family of Lebesgue - measurable functions.
 \item Assume that $cov(\bbk)=2^\omega.$ 
       There exists continuum many different $(2^\omega,\bbk,\bbl)$ - Luzin sets which 
       aren't equivalent with respect to the family of Baire - measurable functions.
\end{enumerate}
\end{cor}
\begin{proof}
Let us notice that the equality $cov(\bbl)=2^\omega$ implies that 
$2^\omega=cov(\bbl)=cof(\bbl)=non(\bbk).$ 
Similarly, the equality $cov(\bbk)=2^\omega$ implies that 
$2^\omega=cov(\bbk)=cof(\bbk)=non(\bbl)$ (see \cite{cichondiagram}).
Corollary \ref{conti2} finishes the proof.  
\end{proof}

\section{Luzin sets and forcing}
Now, let us focus on the class of forcings which preserves being $(\ci,\cj)$-Luzin set. Let us start with a technical observation.

\begin{lemma}\label{jacek} Assume that  $\ci$ has Fubini property. 
Suppose that $\bbp_\ci=\Borel(\cx)\setminus \ci$ is a proper definable forcing.  
Let $B\in \ci$ be a set in $V^{\bbp_\ci}[G].$ Then $B\cap \cx^V\in \ci$. 
\end{lemma}

\begin{proof} Let  $\dot{B}$ -- name for $B$, $\dot{r}$ -- canonical name for 
generic real, $C\subseteq \cx\times\cx$ - Borel set from the ideal $\ci.$
$C$ is coded in ground model $V$ and $B=C_{\dot{r}_G}.$ 

Now by Fubini property:
$$
\{ x:\;\; C^x\notin \ci\}\in \ci.
$$
Let $x\in B\cap\cx^V$ then $V[G]\models x\in B$
$$
0<\lbv x\in \dot{B}\rbv =\lbv x\in C_{\dot{r}}\rbv =\lbv (\dot{r},x)\in C\rbv=\lbv\dot{r}\in C^x\rbv =[C^x]_\ci
$$
Then we have: 
$$
B\cap\cx^V\subseteq\{ x: C^x\notin \ci\}\in \ci.
$$
But the last set is coded in ground model because the set $C$ was coded in $V$.
\end{proof}

\begin{theorem} Assume that $\omega<\kappa$ and $\ci, \cj$ are c.c.c. and have Fubini property. Suppose that $\bbp_\ci=\Borel(\cx)\setminus \ci$ and $\bbp_\cj=\Borel(\cx)\setminus \cj$ are definable forcings. Then $\bbp_\cj$ preserves $(\kappa,\ci,\cj)$ - Luzin set property.
\end{theorem}
\begin{proof} Let $L$ be a $(\kappa,\ci,\cj)$ - Luzin set in $V$. In $V[G]$ take any $B\in \ci$ then $L\cap B\cap V=L\cap B$ but $L\cap B\in \ci$ in $V$ so $L\cap B\in \cj$ in $V$ by definition of $L.$ Finally, by Lemma \ref{jacek}
$$
L\cap B=L\cap B\cap V\in \cj\text{ in } V[G].
$$
\end{proof}

\begin{theorem} 
 Let $(\bbp,\le)$ be a forcing notion such that
 $$
 \{ B : B\in\ci\cap \Borel(\cx), B \text{ is coded in } V\}
 $$  
 is a base for $\ci$ in $V^\bbp[G].$ 
 Assume that Borel codes for sets from ideals $\ci,\cj$ are absolute. 
 Then $(\bbp,\le)$ preserve being $(\ci,\cj)$ - Luzin sets.
\end{theorem}
\begin{proof}
Let $L$ be a $(\ci,\cj)$ - Luzin set in ground model $V.$
We will show that $V^\bbp[G]\models L\text{ is }(\ci,\cj)\text{ - Luzin set}.$

Let us work in $V^\bbp[G].$ Fix $I\in\ci.$ $\ci$ has Borel base consisting of sets 
coded in $V$. So, there exists 
$b\in\omega^\omega\cap V$ such that $I\subseteq\#b\in\ci.$ 

By absoluteness of Borel codes from $\ci$ we have that $V\models\#b\in\ci.$
$L$ is a $(\ci,\cj)$ - Luzin set in the model $V.$ So, there is 
$c\in \omega^\omega\cap V$ which codes Borel set from the ideal $\cj$ 
such that $V\models L\cap\#b\subseteq\#c.$ By absoluteness of Borel codes from $\cj$ we get that
$$
V^\bbp[G]\models L\cap B\subseteq L\cap\#b\subseteq\#c\in\cj,
$$ 
what proves that $L$ is a $(\ci,\cj)$ - Luzin set in generic extension.
\end{proof}

The above theorem gives us a series of corollaries.
\begin{cor}\label{same-reals} 
 Let $(\bbp,\le)$ be any forcing notion which does not 
 change the reals i. e. $(\omega^\omega)^V=(\omega^\omega)^{V^\bbp[G]}.$ 
 Assume that Borel codes for sets from ideals $\ci,\cj$ are absolute. 
 Then $(\bbp,\le)$ preserve being $(\ci,\cj)$ - Luzin sets.
\end{cor}


\begin{cor}\label{wniosek1} Assume that $(\bbp,\le)$ is a $\sigma$-closed forcing and Borel codes for sets from ideals $\ci,\cj$ are absolute. Then $(\bbp,\le)$ preserve $(\ci,\cj)$ - Luzin sets.
\end{cor}

\begin{cor}\label{iteracja} Let $\lambda\in On$ be an ordinal number. 
Let $\bbp_\lambda=\< (P_\alpha,\dot{Q}_\alpha):\;\;\alpha<\lambda\>$ be iterated 
forcing with countable support. Spouse that 
\begin{enumerate}
 \item for any $\alpha<\lambda$ $P_\alpha\force \dot{Q}_\alpha - \sigma\text{ closed }$,
 \item Borel codes for sets from ideals $\ci, \cj$ are absolute,
\end{enumerate}
then $\bbp_\lambda$ preserve $(\ci, \cj)$ - Luzin sets.
\end{cor}
\begin{proof} Our forcing $\bbp_\lambda$ is $\sigma$-closed because it is countable support iteration of $\sigma$ -closed forcings. So, we can apply Corollary~\ref{wniosek1}
to finish the proof.
\end{proof}

Now, let us consider some properties of countable support iteration connected with preservation of some relation.
We will follow notation given by Goldstern (see \cite{goldi}).

First, let us consider measure case. Let $\Omega$ is a family of clopen sets of Cantor space $2^\omega$ and
$$
C^{random}=\{ f\in \Omega^\omega:(\forall n\in\omega) \mu(f(n))<2^{-n}\}
$$ 
with discrite topology. 
If $f\in C^{random}$ then let us define the following set $A_f=\bigcap_{n\in\omega}\bigcup_{k\ge n} f(k)$.

Now, we are ready to define 
the following relation $\sqsubseteq=\bigcup_{n\in\omega} \sqsubseteq_n$ where
$$
(\forall f\in C^{random})(\forall g\in 2^\omega) (f\sqsubseteq_n g\iff (\forall k\ge n)\; g\notin f(k)).
$$

Definition of the notion of preservation of relation $\sqsubseteq^{random}$ by forcing notion $(\bbp,\le)$ can be found in paper \cite{goldi}. Let us focus on the following consequence of that definition.

\begin{fact}[Goldstern]\label{outer} If $(\bbp,\le)$ preserves $\sqsubseteq^{random}$ then $\bbp\force \mu^*(2^\omega\cap V)=1$.
\end{fact}

Now, we say that forcing notion $\bbp$  preserves outer measure iff $\bbp$ preserves $\sqsubseteq^{random}$.

It is well known that Laver forcing preserves some stronger property than $\sqsubseteq^{random}$ (see \cite{judah}).
So, Laver forcing preserves outer measure.

In \cite{goldi} we can find the following theorem:
\begin{theorem}[Goldstern]\label{goldi_iteration} Let $\bbp_\lambda=((P_\alpha,Q_\alpha):\;\alpha<\gamma)$ be any countable support iteration such that
$$
(\forall \alpha<\gamma)\; P_\alpha\force Q_\alpha \text{ preserves }\sqsubseteq^{random} 
$$
then $\bbp_\gamma$ preserves the relation $\sqsubseteq^{random}$.
\end{theorem}

\begin{theorem}\label{random_luzin}
 Assume that $\bbp$ is a forcing notion which preserves $\sqsubseteq^{random}$. Then $\bbp$ preserves being 
 $(\bbl, \bbk)$-Luzin set.
\end{theorem}
\begin{proof}
 Assume that $V\models L \mbox{ is } (\bbl,\bbk)\mbox{-Luzin set.}$ Let us work in $V^\bbp[G].$
 Take any null set $A\in\bbl.$ Then there is a null set $B$ in ground model such that $A\cap V\subseteq B.$
 
Indeed, let us assume that there is no such $B\in V.$ Then without loss of generality 
$(2^\omega\setminus A)\cap V\in\bbl.$ But $A\in\bbl$ then we have that 
$2^\omega\cap V\subset A\cup ((2^\omega\setminus A)\cap V)$ which is a null set. 
But by Fact \ref{outer} $\mu^*(2^\omega\cap V)=1.$ So we have a contradiction.

Then intersection $A\cap L\subseteq B\cap L\in\bbk$ is a meager set in ground model. 
Then by absolutnes of borel codes of meager sets the set $A\cap L$ is a meager set what finishes the proof.
\end{proof}

\begin{remark} In constructible universe $L$ let us consider the countable forcing iteration $P_{\omega_2}=((P_\alpha,Q_\alpha):\alpha<\omega_2)$ of the length $\omega_2$ as follows, for any $\alpha<\omega_2$
\begin{itemize}
 \item if $\alpha$ is even then $P_\alpha\force ''Q_\alpha \text{ is random forcing}''$,
 \item in odd case $P_\alpha\force ``Q_\alpha\text{ is Laver forcing}''$.
\end{itemize}
Previously we noticed that both random and Laver forcing, preserves $\sqsubseteq^{random}$ and then by 
 Theorem \ref{goldi_iteration}  $P_{\omega_2}$ preserves relation $\sqsubseteq^{random}.$ By Theorem \ref{random_luzin} 
the $(\bbl,\bbk)$-Luzin sets are preserved by our iteration $P_{\omega_2}$. 
Moreover, in generic extension we have $cov(\bbl)=\omega_2$ and $2^\omega=\omega_2$ (for details see \cite{goldi}). 

Asuume that in the ground model $A$ is $(\bbl,\bbk)$-Luzin set with outer measure equal to one. 
Then in generic extension it has outer measure one and $|A|=\omega_1.$ So, it  does not contain any Lebesgue positive 
Borel set. Thus $A$ is completely $\bbl$-nonmeasurable set.
\end{remark}

The analogous machinery can be used for ideal of meager sets $\bbk$. Let us recall  the necessary definitions (see \cite{goldi}). 

Let $C^{Cohen}$ be set of all functions from $\omega^{<\omega}$ into itself. 
Then $\sqsubseteq^{Cohen}=\bigcup_{n\in\omega} \sqsubseteq_n^{Cohen}$ and for any $n\in\omega$ let
$$
(\forall f\in C^{Cohen})(\forall g\in \omega^\omega)( f\sqsubseteq_n^{Cohen} g\text{ iff }
(\forall k<n)( g\upharpoonright k^{\frown} f(g\upharpoonright k)\subseteq g)).
$$

Then finally we have the following theorem:
\begin{theorem}\label{cohen_luzin}
 Assume that $\bbp$ is a forcing notion which preserves $\sqsubseteq^{Cohen}$. Then $\bbp$ preserves being 
 $(\bbk, \bbl)$-Luzin set.
\end{theorem}

The another preservation theorem which is due to Shelah (see \cite{shelah} and also \cite{schliri}) is as follows
\begin{theorem}[Shelah]\label{shelah_iteration}
Let $\bbp_\lambda=((P_\alpha,\dot{Q}_\alpha):\;\alpha<\lambda)$ be any countable support iteration such that 
$(\forall \alpha<\gamma)$ $P_\alpha\force Q_\alpha \text{is proper } $ and
$$
P_\alpha\force Q_\alpha\force \text{every new open dense set contains old open dense set}
$$
then $\bbp_\lambda \force \text{ every new open dense contains old open dense set}$.
\end{theorem}

We can easily derive
\begin{cor} Let $\bbp_\lambda=((P_\alpha,\dot{Q}_\alpha):\;\alpha<\lambda)$ be any countable support iteration such that 
$(\forall \alpha<\lambda)$ $P_\alpha\force Q_\alpha \text{is proper } $ and
$$
P_\alpha\force Q_\alpha\force \text{every new open dense set contains old open dense set}
$$ 
Then $\bbp_\lambda$ preserves being $(\bbk,\bbl)$-Luzin set.
\end{cor}

\end{document}